\documentclass[11pt]{amsart}

\usepackage{amsmath}
\usepackage{amsfonts}
\usepackage{amssymb}
\usepackage{amscd}
\usepackage{amsthm}
\usepackage{filecontents}
\usepackage{framed}
\usepackage{fullpage}
\usepackage{latexsym}

\usepackage{hyperref}
\hypersetup{colorlinks,linkcolor={red},citecolor={blue},urlcolor={blue}}

\evensidemargin\oddsidemargin

\newtheorem{theorem}{Theorem}[section]
\newtheorem{lemma}[theorem]{Lemma}

\newtheorem{proposition}[theorem]{Proposition}

\theoremstyle{definition}
\newtheorem{definition}[theorem]{Definition}

\newtheorem{remark}[theorem]{Remark}

\newtheorem{conjecture}[theorem]{Conjecture}

\numberwithin{equation}{section}

\newcommand{\CC}{\mathbb C}
\newcommand{\HH}{\mathbb H}

\newcommand{\NN}{\mathbb N}
\newcommand{\PP}{\mathbb P}
\newcommand{\QQ}{\mathbb Q}
\newcommand{\RR}{\mathbb R}
\newcommand{\ZZ}{\mathbb Z}

\newcommand{\cD}{\mathcal D}
\newcommand{\cH}{\mathcal H}

\newcommand{\SL}{\mathop{\mathrm {SL}}\nolimits}

\newcommand{\SO}{\mathop{\mathrm {SO}}\nolimits}

\newcommand{\Orth}{\mathop{\null\mathrm {O}}\nolimits}

\newcommand{\rank}{\mathop{\mathrm {rank}}\nolimits}

\newcommand{\latt}[1]{{\langle{#1}\rangle}}

\newcommand{\Borch}{\operatorname{Borch}}
\def\div{\operatorname{div}}

\newcommand{\im}{\operatorname{Im}}

\newcommand{\mult}{\operatorname{mult}}

\newcommand{\II}{\operatorname{II}}

\newcommand{\pr}{\operatorname{pr}}
\newcommand{\Proj}{\operatorname{Proj}}
\newcommand{\Spm}{\operatorname{Spm}}

\begin{document}

\title[Free algebras of orthogonal modular forms]{The classification of free algebras of orthogonal modular forms}

\author{Haowu Wang}

\address{Max-Planck-Institut f\"{u}r Mathematik, Vivatsgasse 7, 53111 Bonn, Germany}

\email{haowu.wangmath@gmail.com}


\date{\today}


\begin{abstract}
We prove a necessary and sufficient condition for the graded algebra of automorphic forms on a symmetric domain of type IV being free. From the necessary condition, we derive a classification result.
Let $M$ be an even lattice of signature $(2,n)$ splitting two hyperbolic planes. Suppose $\Gamma$ is a subgroup of the integral orthogonal group of $M$ containing the discriminant kernel. It is proved that there are exactly 26 groups $\Gamma$ such that the space of modular forms for $\Gamma$ is a free algebra.  Using the sufficient condition, we recover some  well-known results. 
\end{abstract}

\maketitle

\section{Introduction}
Let $\Gamma$ be a discrete automorphism group of a complex symmetric domain $\cD$ with fundamental domain of finite volume acting on the affine cone over $\cD$. The space of automorphic forms on $\cD$ for $\Gamma$ is an infinite counterpart of the polynomial invariants of a finite linear group. The seminal Shephard-Todd-Chevalley theorem (\cite{ST54, Che55}) asserts that the algebra of invariants of a finite linear group acting on a complex vector space is free if and only if this group is generated by (complex) reflections.  Similarly, a topological argument in \cite{VP89} shows that if the algebra of automorphic forms is free then the group $\Gamma$ is generated by reflections. It is known that reflections exist only in two infinite families of symmetric domains: complex balls and symmetric domains of type IV in Cartan's classification. In this paper we focus on the latter case which corresponds to orthogonal modular forms, namely automorphic forms on symmetric domains of type IV for orthogonal groups of signature $(2,n)$. 

It is a difficult problem to determine the structure of the algebra of automorphic forms in general. From a geometric perspective, this is equivalent to find a projective model of the modular variety.  If the algebra of modular forms for a congruence group $\Gamma$ acting on $\cD$ is freely generated, then the Satake-Baily-Borel compactification of the modular variety $\cD/\Gamma$ is a weighted projective space (see \cite{BB66}).  In 1962, Igusa proved that the algebra of even-weight Siegel modular forms of genus 2 is freely generated by forms of weights 4, 6, 10, 12 in \cite{Igu62}. Siegel modular forms of genus 2 can be realized as modular forms for the orthogonal group $\Orth(2,3)$. This is the first example of free algebras of orthogonal modular forms in dimension larger than 2. After Igusa, more free algebras of $\Orth(2,n)$-modular forms were constructed in \cite{AI05, DK03, DK06, Kri05, FH00, FS07, HU14, Vin10, Vin18}. Recently, the author proved joint with B. Williams that the spaces of orthogonal modular forms are free algebras for 25 groups in a universal way in \cite{WW20}.

It is another interesting problem to derive some classification of free algebras of orthogonal modular forms. There are only two known results in this direction. The first is attributed to Shvartsman and Vinberg, who proved in  \cite{SV17} that the algebra of modular forms for orthogonal groups of signature $(2,n)$ with $n>10$ is never free. They concluded the result from a criterion of smoothness at infinity for the quotient space of the affine cone over $\cD$ by $\Gamma$. We will give a simple proof of their result in a particular case (see Theorem \ref{th:nonfree}). The second is due to Stuken, who gave a classification of free algebras of Hilbert modular forms which can be realized as orthogonal modular forms of signature (2,2) in \cite{Stu19}. In this paper we present a classification of free algebras of orthogonal modular forms under a mild condition. The idea starts with the Rankin-Cohen-Ibukiyama differential operators introduced in \cite{AI05}, which can be regarded as the Jacobian determinant of $n+1$ modular forms on $\cD$ for $\Gamma$ (see Theorem \ref{th:Jacobian}). Following Vinberg's insights \cite{Vin13}, we are able to prove the following result which gives a necessary and sufficient condition for the graded algebra of modular forms for $\Gamma$ to be free.

\begin{theorem}[Theorem \ref{th:freeJacobian} and Theorem \ref{th:converseJacobian}]
\label{mth1}
If the graded algebra $M_*(\Gamma)$ of modular forms for $\Gamma$ is free, then the Jacobian determinant of the $n+1$ free generators defines a cusp form for $\Gamma$ with the determinant character which vanishes exactly on all mirrors of reflections in $\Gamma$ with multiplicity one. Conversely, if there exists a modular form with a character for $\Gamma$ which vanishes exactly on all mirrors of reflections in $\Gamma$ with multiplicity one and equals a Jacobian determinant of $n+1$ modular forms for $\Gamma$, then $M_*(\Gamma)$ is a free algebra.
\end{theorem}

The sufficient condition provides a powerful method for constructing free algebras of orthogonal modular forms.  We discuss two famous examples of signature $(2,3)$. For full Siegel modular forms of genus 2, the Jacobian determinant of four generators of weights 4, 6, 10, 12 is indeed the unique Siegel modular form of odd weight 35. For Siegel modular forms of genus 2 for the subgroup $\Gamma_2(2,4)$, the Jacobian determinant of four second order theta constants is exactly Igusa's cusp form $\chi_5$ which is the product of ten theta constants and vanishes precisely on the diagonal of the Siegel upper half-plane with multiplicity one. Obviously, the Igusa theorem in \cite{Igu62} and the Runge theorem in \cite{Run93} can be recovered quickly using our result.

The modular form with special divisor in Theorem \ref{mth1} is called reflective in the literature. Reflective modular forms have many applications in algebra and geometry, and the number of such modular forms is finite (see \cite{Ma18}). In \cite{Wan18, Wan19} the author developed an approach to classify reflective modular forms based on the theory of Jacobi forms of lattice index (see \cite{CG13}). Applying this approach to the present case, we find that if $M_*(\Gamma)$ is a free algebra then $\Gamma$ corresponds to a root system of the same rank as $L$ and the Coxeter numbers of the irreducible components of the root system satisfy certain conditions. We then deduce the following theorem from these conditions.

\begin{theorem}[Theorem \ref{th:classification}]
\label{mth2}
Let $M=2U\oplus L(-1)$ be an even lattice of signature $(2,n)$ splitting two hyperbolic planes. Let $\Orth^+(M)$ be the orthogonal group preserving $M$ and the domain $\cD$. Let $\widetilde{\Orth}^+(M)$ be the discriminant kernel which is a subgroup  of $\Orth^+(M)$ acting trivially on the discriminant group of $M$.  Suppose $\Gamma<\Orth^+(M)$ is a  subgroup containing $\widetilde{\Orth}^+(M)$. If $M_*(\Gamma)$ is a free algebra, then $\Gamma$ must be one of the $26$ groups defined as $\Gamma_R=\latt{\widetilde{\Orth}^+(2U\oplus L_R(-1)), W(R)}$, where $R$ is a root system of type $A_r (1\leq r \leq 7)$, $B_r (2\leq r \leq 4)$, $D_r (4\leq r \leq 8)$, $C_r (3\leq r \leq 8)$, $G_2$, $F_4$, $E_6$, $E_7$, or $E_8$, $W(R)$ is the Weyl group of $R$, and $L_R$ is the root lattice generated by $R$ (we rescale its bilinear form by $2$ such that it is even when $L_R$ is an odd lattice).
\end{theorem}

\begin{remark}
The algebra of modular forms for every $\Gamma_R$ above is free and the constructions of generators are known.  It was proved in \cite{HU14} that the algebra of modular forms on $\Orth^+(2U\oplus E_8(-1))$ is freely generated by forms of weights 4, 10, 12, 16, 18, 22, 24, 28, 30, 36, 42, and in \cite{DKW19} that the generators can be chosen as additive lifts of Jacobi Eisenstein series.  The other 25 cases were proved in a universal and elementary way in \cite{WW20}. A general rule characterizing the weights of generators was also given.
\end{remark}

The paper is organized as follows. In the next section we introduce some necessary materials about orthogonal modular forms and Jacobi forms. In \S \ref{sec:condition} we prove the necessary condition in Theorem \ref{mth1}.  \S \ref{sec:classification} is devoted to the proof of Theorem \ref{mth2}. In \S \ref{sec:sufficient} we prove the sufficient condition and present many applications.

\section{Automorphic forms on symmetric domains of type IV}
In this section we give an overview of the theory of orthogonal modular forms. We recall some basic properties of orthogonal modular forms and introduce the theory of Jacobi forms of lattice index and the Rankin-Cohen-Ibukiyama type differential operators.

\subsection{Modular forms for orthogonal groups}
Let $\RR^{2,n}$ be a pseudo-Euclidean vector space of signature $(2,n)$ and let $\Orth_{2,n}$ be the group of its orthogonal transformations. In this paper we always assume that $n\geq 3$. We set $\CC^{2,n}=\RR^{2,n}\otimes \CC$ and consider the cone
$$
\widetilde{\mathcal{L}_n}=\{ \mathcal{Z} \in \CC^{2,n} :  (\mathcal{Z}, \mathcal{Z})=0, (\mathcal{Z},\bar{\mathcal{Z}}) > 0 \}.
$$
It has two complex conjugate connected components. We denote by $\mathcal{L}_n$ one of these components. Let $\Orth^+_{2,n}< \Orth_{2,n}$ be the subgroup of index 2 preserving the component $\mathcal{L}_n$. Let $\cD_n$ be the projectivization of $\mathcal{L}_n$, which is identified with the Hermitian symmetric domain of type IV, namely $\Orth^+_{2,n}/ (\SO_2\times \Orth_n)$. 

Let $\Gamma< \Orth^+_{2,n}$ be an arithmetic subgroup. By \cite[Proposition 5.4]{GHS13}, there exists an even lattice $M$ of signature $(2,n)$ such that $\Gamma$ is contained in $\Orth^+(M)$ which is the orthogonal group fixing $M$. In this paper we are more willing to change the lattice, so we assume that $\Gamma$ is a finite index subgroup of some $\Orth^+(M)$. 

\begin{definition}
Let $k$ be a non-negative integer. A modular form of weight $k$ and character $\chi: \Gamma\to \CC^*$ for $\Gamma$ is a holomorphic function $F: \mathcal{L}_n\to \CC$ satisfying
\begin{align*}
F(t\mathcal{Z})&=t^{-k}F(\mathcal{Z}), \quad \forall t \in \CC^*,\\
F(g\mathcal{Z})&=\chi(g)F(\mathcal{Z}), \quad \forall g\in \Gamma.
\end{align*}
\end{definition}

Geometrically, the modular form $F$ can be viewed as a $\Gamma$-invariant holomorphic section of the $k$-th power of the line bundle $\bar{\pi}$, where $\bar{\pi}$ is the line bundle obtained from the natural $\Orth^+_{2,n}$-invariant holomorphic $\CC^*$-bundle $\pi: \mathcal{L}_n \to \cD_n$ by filling in the zero section. 

By \cite{Bor95}, the modular form $F$ either has weight 0 in which case it is constant, or has weight at least $n/2-1$. The minimal possible positive weight $n/2-1$ is called the singular weight.

The group $\Gamma$ acts properly discontinuously on $\cD_n$, but in general there are elements of finite order in $\Gamma$ which have fixed points in $\cD_n$. This leads to singularities of $\cD_n/\Gamma$.
The quotient $\cD_n/\Gamma$ is a normal complex space and is not compact. In fact, it is a quasi-projective variety of dimension $n$ by \cite{BB66}.  
In order to compactify this quotient, we need to add some boundary components such that the resulting space is a projective variety. The Satake-Baily-Borel compactification provides such a way. In our case, the Satake-Baily-Borel compactification $(\cD_n/\Gamma)^*$ contains $\cD_n/\Gamma$ as a Zariski open subset and it is obtained by adding the following rational boundary components
$$
\coprod_{c} Q_c \bigsqcup \coprod_{P} X_P, 
$$
where $c$ and $P$ run through representatives of the finitely many $\Gamma$-orbits of isotropic lines and isotropic planes in $M\otimes \QQ$ respectively. Each $X_P$ is a modular curve, each $Q_c$ is a point, and $Q_c$ is contained in the closure of $X_Q$ if and only if the representatives may be chosen such that $c\subset P$. Usually, $X_P$ and $Q_c$ are also respectively called 1-dimensional and 0-dimensional boundary components (or cusps).  By Koecher's principle, a modular form is also holomorphic on the boundary, and it is called a \textit{cusp} form if it vanishes on every such boundary component. 

The space of modular forms of weight $k$ with trivial character is a finite-dimensional vector space. We denote this space by $M_k(\Gamma)$. The graded algebra
$$
M_*(\Gamma)=\bigoplus_{k=0}^\infty M_k(\Gamma)
$$
is known to be finitely generated.  The projective variety $\Proj (M_*(\Gamma))$ coincides with the above Satake-Baily-Borel compactification of $\cD_n/\Gamma$ (see \cite{BB66}).

\subsection{Fourier expansion of orthogonal modular forms}
We first fix some notations. For an even lattice $M$, we denote its dual lattice by $M^\vee$ and its discriminant group by $D(M)=M^\vee/M$. The discriminant kernel $\widetilde{\Orth}^+(M)$ is defined as the kernel of the reduction map $\Orth^+ (M) \to \Orth(D(M))$. The level of $M$ is the minimal positive integer $N$ such that $N(x,x)\in 2\ZZ$ for all $x\in M^\vee$. For $v\in M$, we denote the positive generator of the ideal $(v,M)=\{(v,x): x\in M\}$ by $\div(v)$. 
For any integer $a$, the lattice obtained by rescaling $M$ with $a$ is denoted by $M(a)$.

Let $M$ be an even lattice of signature $(2,n)$ with $n\geq 3$. By \cite[P.43]{Ser73}, $M$ has isotropic vectors. Moreover, $M$ contains an isotropic plane when $n\geq 5$. 
Let $c$ be a primitive isotropic vector of $M$, i.e. a zero-dimensional cusp. We introduce the Fourier expansion of orthogonal modular forms at the cusp $c$ following \cite{Bor95, CG13}. 

For any $\mathcal{Z}\in \mathcal{L}_n$ there exists a unique $\alpha\in \CC^*$ such that $(\alpha \mathcal{Z},c)=1$. It follows that
$$
\cD(M)_c=\left\{\mathcal{Z}\in \mathcal{L}_n: (\mathcal{Z},c)=1\right\}\cong \cD(M):=\cD_n.
$$
The lattice 
$
M_c=c^{\perp}/c
$
is an even lattice of signature $(1,n-1)$. We fix an element $b\in M^\vee$ such that $(c,b)=1$. Then one has
$
M_c\cong M_{c,b}=M\cap c^\perp\cap b^\perp.
$
This yields a decomposition
$$
M\otimes \QQ=M_{c,b}\otimes\QQ \oplus (\QQ b+\QQ c).
$$
Using the hyperbolic lattice $M_c\otimes\RR$ we can define a positive cone
$$
C(M_c)=\{ X \in M_c\otimes \RR : (X,X)>0\}.
$$
Let $C^+(M_c)$ be one of the two connected components of $C(M_c)$.  The following tube domain gives the complexification of $C^+(M_c)$
\begin{equation}
\mathcal{H}_c(M)=M_c\otimes \RR + iC^+(M_c).
\end{equation}
There is an isomorphism $\pr_c: \mathcal{H}_c(M)\to \cD(M)_c\cong \cD(M)$ defined as
\begin{equation}
\pr_c: \quad Z\mapsto Z\oplus \left[b-\frac{(Z,Z)+(b,b)}{2}c  \right].
\end{equation}
Using the coordinate  $Z\in\mathcal{H}_c(M)$ determined by $c$ and $b$, we identify an arbitrary orthogonal modular form $F$ of weight $k$ with a modular form $F_{c,b}$ (or simply $F_c$) on the tube domain $\mathcal{H}_c(M)$:
\begin{equation}
F_{c,b}(Z)=F(\pr_c(Z)).
\end{equation}
For every $g\in \Orth^+(M)$ and $Z\in \mathcal{H}_c(M)$, there exist  $J_{c,b}(g,Z)\in \CC^*$ and $g\latt{Z} \in \mathcal{H}_c(M)$ such that 
\begin{equation}
g\pr_c (Z)= J_{c,b}(g,Z) \pr_c (g\latt{Z}).
\end{equation}
The above relation defines an action of  $\Orth^+(M)$ on $\mathcal{H}_c(M)$. A modular form of weight $k$ and character $\chi$ can be also defined on the tube domain via 
\begin{align*}
F_{c,b}\lvert_k g &=\chi(g)F_{c,b},\\
(F_{c,b}\lvert_k g)(Z):&=J_{c,b}(g,Z)^{-k}F_{c,b}(g\latt{Z}).
\end{align*}

Let $F\in M_k(\widetilde{\SO}^+(M))$.  Since the \textit{Eichler transvection} 
\begin{equation}\label{def:Eichler}
t(c,a): v\mapsto v-(a,v)c+(c,v)a-\frac{1}{2}(a,a)(c,v)c
\end{equation}
belongs to $\widetilde{\SO}^+(M)$ for all $a\in M_{c,b}$ and $t(c,a)(\pr_c (Z))= \pr_c (Z+a)$,
we have $F_c(Z+a)=F_c(Z)$, which gives the Fourier expansion of $F$ at the cusp $c$:
\begin{equation}
F_c(Z)=\sum_{l\in M_{c,b}^\vee} f(l)\exp (2\pi i (l,Z)).
\end{equation}
The Koecher principle shows that the function $F_c$ is holomorphic at the cusp $c$, which implies that if
$f(l)\neq 0$ then $l$ belongs to the closure of $C^+(M_c)$. If a modular form has singular weight then all the Fourier coefficients associated to vectors of non-zero norm vanish.

\begin{remark}
Let $\Gamma$ be a finite index subgroup of $\Orth^+(M)$. By \cite[\S 9]{PR94}, $\Gamma$ is a congruence subgroup, namely there exists a positive integer $d$ such that $\widetilde{\Orth}^+(M(d))<\Gamma$.  Thus for any arithmetic subgroup $\Gamma$, there is an even lattice $M_1$ such that $\widetilde{\Orth}^+(M_1)<\Gamma<\Orth^+(M_1)$. Hence we have the Fourier expansion of the above form for any modular forms.
\end{remark}

\subsection{Fourier-Jacobi expansion: Jacobi forms of lattice index} 
If the lattice $M_{c,b}$ also contains an isotropic vector, then one has the Fourier-Jacobi expansion of modular forms. We explain this precisely. Assume that $M$ contains two hyperbolic planes, i.e. $M=U_1\oplus U_2 \oplus L(-1)$, where $L$ is an even positive definite lattice and $U_i=\ZZ e_i+\ZZ f_i$, $(e_i,e_i)=(f_i,f_i)=0$, $(e_i,f_i)=1$. We fix $(e_1,e_2,...,f_2,f_1)$ as a basis of $M$, where $...$ stands for a basis of $L$.  We choose $c=e_1$ and $b=f_1$. Then the tube domain $\cH_{e_1}(M)$ can be written as 
$$
\cH(L)=\{Z=(\tau,\mathfrak{z},\omega)\in \HH\times (L\otimes\CC)\times \HH: 
(\im Z,\im Z)>0\}, 
$$
where $(\im Z,\im Z)=2\im \tau \im \omega - (\im \mathfrak{z},\im \mathfrak{z})_L$. Thus $F_{e_1}=F_{e_1,f_1}$ has the following expansion
\begin{align*}
F_{e_1}(Z)&=\sum_{m=0}^\infty\sum_{\substack{n\in \NN\\ 2nm\geq (\ell,\ell)}}f(n,\ell,m)e^{2\pi i (n\tau+(\ell,\mathfrak{z})+m\omega)}\\
&=\sum_{m=0}^\infty\phi_m(\tau,\mathfrak{z})e^{2\pi i m\omega}.
\end{align*}

Let $\Gamma^J(L)$ be the subgroup of $\Orth^+ (M)$ preserving the above Fourier-Jacobi expansion. This group is called the Jacobi group and can be realized as the semi-direct product of $\SL_2(\ZZ)$ with the integral Heisenberg group of $L$. We define Jacobi forms as modular forms with respect to the Jacobi group. 

\begin{definition}
For $k\in\ZZ$, $t\in \NN$, a holomorphic 
function $\varphi : \HH \times (L \otimes \CC) \rightarrow \CC$ is 
called a {\it weakly holomorphic} Jacobi form of weight $k$ and index $t$ associated to $L$,
if it satisfies the following transformation laws
\begin{align*}
\varphi \left( \frac{a\tau +b}{c\tau + d},\frac{\mathfrak{z}}{c\tau + d} 
\right)& = (c\tau + d)^k 
\exp{\left(i \pi t \frac{c(\mathfrak{z},\mathfrak{z})}{c 
\tau + d}\right)} \varphi ( \tau, \mathfrak{z} ), \quad \left(\begin{array}{cc} 
a & b \\ 
c & d
\end{array}
\right)   \in \SL_2(\ZZ), \\
\varphi (\tau, \mathfrak{z}+ x \tau + y)&= 
\exp{\bigl(-i \pi t ( (x,x)\tau +2(x,\mathfrak{z}))\bigr)} 
\varphi (\tau, \mathfrak{z} ), \quad x,y \in L,
\end{align*}
and if its Fourier expansion  takes the form
\begin{equation*}
\varphi ( \tau, \mathfrak{z} )= \sum_{n\geq n_0 }\sum_{\ell\in L^\vee}f(n,
\ell)q^n\zeta^\ell,
\end{equation*}
where $n_0$ is a constant, $q=e^{2\pi i \tau}$ and $\zeta^\ell=e^{2\pi i (\ell,
\mathfrak{z})}$. 
If $f(n,\ell) = 0$ whenever $n < 0$,
then $\varphi$ is called a {\it weak} Jacobi form.  If $ f(n,\ell) = 0 $ whenever
$ 2nt - (\ell,\ell) < 0 $ (resp. $\leq 0$),
then $\varphi$ is called a {\it holomorphic}
(resp. {\it cusp}) Jacobi form. 
\end{definition}

We denote by $J^{!}_{k,L,t}$ (resp. $J^{w}_{k,L,t}$, $J_{k,L,t}$, 
$J_{k,L,t}^{\text{cusp}}$) the vector space of weakly holomorphic Jacobi 
forms (resp. weak, holomorphic, cusp Jacobi forms) 
of weight $k$ and index $t$ for $L$. 
The classical Jacobi forms defined by Eichler--Zagier \cite{EZ85} $J_{k,N}$
are identical to the Jacobi forms $J_{k,A_1, N}$ for the lattice  $A_1=\latt{\ZZ, 2x^2}$.

By definition, each Fourier-Jacobi coefficient $\phi_m$ is a holomorphic Jacobi form of weight $k$ and index $m$ associated to $L$. For the lattice $M$ not containing $2U$, the similar Fourier-Jacobi coefficients are holomorphic Jacobi forms with respect to some congruence subgroup of $\SL_2(\ZZ)$.

\subsection{Reflective modular forms}
Let $M$ be an even lattice of signature $(2,n)$ with $n\geq 3$. We set $\cD(M)=\cD_n$.
For any $r\in M^\vee$ of negative norm, the hyperplane
\begin{equation}
 \cD_r(M)=r^\perp\cap \cD(M)=\{ [\mathcal{Z}]\in \cD(M) : (\mathcal{Z},r)=0\}
\end{equation}
is called the rational quadratic divisor associated to $r$.
The reflection fixing $\cD_r(M)$ is defined as
\begin{equation}
\sigma_r(x)=x-\frac{2(r,x)}{(r,r)}r,  \quad x\in M.
\end{equation}
The hyperplane $\cD_r(M)$ is called the mirror of $\sigma_r$.
A primitive vector $r\in M$ of negative norm is called reflective if $\sigma_r\in\Orth^+(M)$, in which case we call $\cD_r(M)$ a reflective divisor. For $\lambda \in D(M)$ and $m\in \QQ$, we define
\begin{equation}
\cH(\lambda,m)=\bigcup_{\substack{v \in M+\lambda \\ (v,v)=2m}}  \cD_v(M)
\end{equation}
as the Heegner divisor of discriminant $(\lambda,m)$.

A primitive vector $l\in M$ with $(l,l)=-2d$ is reflective if and only if $\div(l)=2d$ or $d$.  We set $\lambda=[l/\div(l)]\in D(M)$. If $\div(l)=2d$, then $\cD_\lambda(M)$ is contained in $\cH(\lambda,-1/(4d))$. If $\div(l)=d$, then it is contained in 
$$\cH\left(\lambda,-\frac{1}{d}\right)-\sum_{2\nu=\lambda}\cH\left(\nu,-\frac{1}{4d}\right).$$
 
A modular form $F$ for $\Gamma<\Orth^+(M)$ is called reflective if its zero divisor is a sum of some reflective divisors. In particular, $F$ is called 2-reflective if its support of zero divisor is contained in $\cH(0,-1)$, and is called a modular form with complete 2-divisor if $\div(F)=\cH(0,-1)$. Reflective modular forms are very rare (see \cite{Ma17, Ma18, Wan18}) and have many applications on the theory of generalized Kac-Moody algebras, reflection groups and in algebraic geometry (see e.g. \cite{Bor00, Sch06, GN18, Gri18}). We refer to \cite{Sch06, Sch17, Dit19, Wan19} for some classification results of reflective modular forms. 

Borcherds’ singular theta correspondence (see \cite{Bor98} or \cite{Bru02}) is a powerful way to construct modular forms for orthogonal groups. It maps modular forms for the Weil representation to orthogonal modular forms which are often called Borcherds products. By \cite{Bru14}, every reflective modular form for $\widetilde{\Orth}^+(M)$ is a Borcherds product of some vector-valued modular form for the Weil representation of $\SL_2(\ZZ)$ attached to the discriminant form $D(M)$ if $M$ can be represented as $U\oplus U(m)\oplus L(-1)$. In this paper, we use the following variant of Borcherds product due to Gritsenko-Nikulin in the context of Jacobi forms.

\begin{theorem}[Theorem 4.2 in \cite{Gri18}]
\label{th:Borcherds}
We fix an ordering $\ell >0$ in $L^\vee$ in a way similar to positive root systems (see the bottom of page $825$ in \cite{Gri18}). Let 
$$
\varphi(\tau,\mathfrak{z})=\sum_{n\in\ZZ, \ell\in L^\vee}f(n,\ell)q^n 
\zeta^\ell \in J^{!}_{0,L,1}.
$$
Assume that $f(n,\ell)\in \ZZ$ for all $2n-(\ell,\ell)\leq 0$. 
There is a meromorphic modular form of weight 
$\frac{1}{2}f(0,0)$ and character $\chi$ with respect to  
$\widetilde{\Orth}^+(2U\oplus L(-1))$ defined as
$$
\Borch(\varphi)(Z)=q^A \zeta^{\vec{B}} \xi^C\prod_{\substack{n,m\in\ZZ, \ell
\in L^\vee\\ (n,\ell,m)>0}}(1-q^n \zeta^\ell \xi^m)^{f(nm,\ell)}, 
$$ 
where $Z= (\tau,\mathfrak{z}, \omega) \in \cH (L)$, 
$q=\exp(2\pi i \tau)$, 
$\zeta^\ell=\exp(2\pi i (\ell, \mathfrak{z}))$, $\xi=\exp(2\pi i \omega)$,
the notation $(n,\ell,m)>0$ means that either $m>0$, or $m=0$ 
and $n>0$, or $m=n=0$ and $\ell<0$, and
$$
A=\frac{1}{24}\sum_{\ell\in L^\vee}f(0,\ell),\quad
\vec{B}=\frac{1}{2}\sum_{\ell>0} f(0,\ell)\ell, \quad C=\frac{1}{2\rank(L)}\sum_{\ell\in L^\vee}f(0,\ell)(\ell,\ell).
$$

The character $\chi$ is induced by the character of the first Fourier-Jacobi coefficient of 
$\Borch(\varphi)$ and by the relation $\chi(V)=(-1)^D$, where
$V: (\tau,\mathfrak{z}, \omega) \to (\omega,\mathfrak{z},\tau)$, 
and $D=\sum_{n<0}\sigma_0(-n) f(n,0)$.

The poles and zeros of $\Borch(\varphi)$ lie on the rational quadratic 
divisors $\cD_v$, where $v\in 2U\oplus L^\vee(-1)$ is a primitive vector 
with $(v,v)<0$. The multiplicity of this divisor is given by 
$$ \mult \cD_v = \sum_{d\in \ZZ,d>0 } f(d^2n,d\ell),$$
where $n\in\ZZ$, $\ell\in L^\vee$ such that 
$(v,v)=2n-(\ell,\ell)$ and $v-(0,0,\ell,0,0)\in 2U\oplus L(-1)$.

The vector $(A, \vec{B}, C)$ is called the Weyl vector of the Borcherds product.
\end{theorem}

\subsection{The Jacobian determinant of orthogonal modular forms}
The following Rankin-Cohen-Ibukiyama differential operators will play a vital role in this paper. It was first introduced in \cite{AI05} for Siegel modular forms. We here prove more properties of this operator for orthogonal modular forms. 

\begin{theorem}\label{th:Jacobian}
Let $M$ be an even lattice of signature $(2,n)$ with $n\geq 3$, and let $\Gamma<\Orth^+(M)$ be a finite index subgroup. Let $f_i\in M_{k_i}(\Gamma)$ for $1\leq i \leq n+1$. We view $f_i$ as modular forms on the tube domain at a given zero-dimensional cusp. Let $z_i$, $1\leq i \leq n$, be the coordinates of the tube domain. We define 
$$
J:=J(f_1,...,f_{n+1})=\left\lvert \begin{array}{cccc}
k_1f_1 & k_2f_2 & \cdots & k_{n+1}f_{n+1} \\ 
\frac{\partial f_1}{\partial z_1} & \frac{\partial f_2}{\partial z_1} & \cdots & \frac{\partial f_{n+1}}{\partial z_1} \\ 
\vdots & \vdots & \ddots & \vdots \\ 
\frac{\partial f_1}{\partial z_n} & \frac{\partial f_2}{\partial z_n} & \cdots & \frac{\partial f_{n+1}}{\partial z_n}
\end{array}   \right\rvert.
$$
\begin{enumerate}
\item The function $J$ is a modular form of weight $n+\sum_{i=1}^{n+1}k_i$ for $\Gamma$ with the character $\det$, where $\det$ is the determinant.
\item The function $J$ is not identically zero if and only if the $n+1$ modular forms $f_i$ are algebraically independent over $\CC$.
\item The function $J$ is a cusp form.
\item Let $r\in M$. If the reflection $\sigma_r$ belongs to $\Gamma$, then $J$ vanishes on the hyperplane $\cD_r(M)$.
\item Assume that $M=2U\oplus L(-1)$ and $\widetilde{\Orth}^+(M)<\Gamma$. We define $J$ at the standard 1-dimensional cusp determined by $2U$. Then the Fourier-Jacobi expansion of $J$ satisfies $J=q^{n-1}\xi^{n-1}(\cdots)$, i.e.
$$
J(Z)=\sum_{\substack{a,t\in \NN\\ a, t\geq n-1}}\sum_{\substack{\ell \in L^\vee \\2at\geq (\ell,\ell)}} f(a,\ell,t)q^a\zeta^\ell\xi^t.
$$
\end{enumerate}
\end{theorem}

\begin{proof} 
\begin{itemize}
\item[(1)] The proof is similar to that of \cite[Proposition 2.1]{AI05}.

\item[(2)] Suppose that $J\neq 0$. If the $n+1$ modular forms $f_i$ are not algebraically independent over $\CC$, then there exists a non-zero polynomial $P$ over $\CC$ in $n+1$ variables such that $P(f_1,...,f_{n+1})=0$. We write 
$$
P(X_1,...,X_{n+1})=\sum_{(i_1,...,i_{n+1})\in \NN^{n+1}} c(i_1,...,i_{n+1}) X_1^{i_1}\cdots X_{n+1}^{i_{n+1}}.
$$
We can assume that $\sum_{j=1}^{n+1} k_j i_j$ is a fixed constant $c$ for any $(i_1,...,i_{n+1})\in \NN^{n+1}$ due to modularity. Considering the differentials of $P(f_1,...,f_{n+1})$ with respect to $z_1$,...,$z_n$ respectively, we obtain the following system of linear equations 
$$
\mathbf{J} \left(\frac{\partial P}{\partial f_1}, \frac{\partial P}{\partial f_2},...,\frac{\partial P}{\partial f_{n+1}}  \right)^t = \left( cP, \frac{\partial P}{\partial z_1},...,\frac{\partial P}{\partial z_n} \right)^t=0,
$$
where $\mathbf{J}$ is the Jacobian matrix in the definition of $J$. This leads to a contradiction. 

Conversely,  if these $f_i$ are algebraically independent over $\CC$, then the $n$ functions $f_2^{k_1}/f_1^{k_2}$, ..., $f_{n+1}^{k_1}/f_1^{k_{n+1}}$ are local parameters of the $n$-dimensional variety $\cD_n/\Gamma$. Therefore their usual Jacobian determinant equal to $J$ up to a non-zero multiple is not identically zero. 

\item[(3)] It suffices to show that $J$ vanishes on every rational boundary component. We first define the  Jacobian on the cone $\mathcal{L}_n$ following \cite[\S 9]{Vin13} such that the definition of $J$ is independent of $0$-dimensional cusps. In this way, it is easier to consider the Fourier expansions of $J$ at any cusps. To this end, we add a function $g_2(\mathcal{Z})=(\mathcal{Z},\mathcal{Z})$ for $\mathcal{Z}\in \CC^{2,n}$. The functions $f_i$ are defined only on an open subset of the hypersurface $\{g_2=0\}$, and the differentials $df_i$ of $f_i$ at a point of $\mathcal{L}_n$ are linear forms on the tangent space of the hypersurface $\{g_2=0\}$. Obviously, we can extend these linear forms to the larger vector space $\CC^{2,n}$. These extensions are not unique but they are defined up to addition of some multiples of the differential form $dg_2$. Let $x_1,x_2,...,x_{n+2}$ be the coordinates of $\CC^{2,n}$. With the help of these extensions of $df_i$, we can consistently define a holomorphic function on $\mathcal{L}_n$ as the usual Jacobian
$$
\hat{J}=\frac{\partial (g_2, f_1, f_2, ..., f_{n+1})}{\partial (x_1,x_2,...,x_{n+2})}.
$$
The function $\hat{J}$ is well defined and independent of the choice of these extensions because the difference between two extensions of $df_i$ is a scalar of $dg_2$ at a given point of $\mathcal{L}_n$.  By direct calculations, we find that the reduction of $\hat{J}$ on the tube domain coincides with the above $J$ up to some non-zero multiple.  We then conclude that the definition of $J$ does not depend on the choice of 0-dimensional cusps up to some non-zero multiple. Besides, the definition of $\hat{J}$ implies the assertion $(1)$ immediately.

We now prove that $J$ vanishes on cusps. At a 0-dimensional cusp, the value of $J$ is equal to the constant term of its Fourier expansion at this cusp.  It is easy to see from the definition of $J$ that the constant term must be zero. At an 1-dimensional cusp, the value of $J$ is given by the Siegel operator. More precisely, it is equal to the zeroth coefficient of the Fourier-Jacobi expansion of $J$ at this cusp, which is a modular form with respect to a congruence subgroup of $\SL_2(\ZZ)$. The partial derivatives with respect to some coordinates will cancel the zeroth Fourier-Jacobi coefficients of these $f_i$. It follows that the value of $J$ at an 1-dimensional cusp is zero.
Therefore $J$ is a cusp form.
\item[(4)] If $\sigma_r\in \Gamma$, then $J(\sigma_r(\mathcal{Z}))=\det(\sigma_r)J(\mathcal{Z})=-J(\mathcal{Z})$. It follows that $J(\mathcal{Z})=-J(\mathcal{Z})$ if $(\mathcal{Z},r)=0$, which yields that $J$ vanishes on the hyperplane $r^\perp$.
\item[(5)] It follows from the number of partial derivatives in the definition of $J$.
\end{itemize}
\end{proof}

\section{Necessary conditions to be free algebras}\label{sec:condition}
In this section we prove some necessary conditions for the space $M_*(\Gamma)$ being a free algebra.

Let $\Gamma < \Orth^+_{2,n}$ be an arithmetic subgroup. The maximal spectrum $\Spm(M_*(\Gamma))$ can be viewed as the ``affine span'' $(\mathcal{L}_n/\Gamma)^*$ of the quotient space $\mathcal{L}_n/\Gamma$. Recall that the weighted projective space $\PP(a_1,...,a_{n+1})$ is defined as the quotient space 
$$
(\CC^{n+1}-\{0\})/ \sim,
$$
where $\sim$ is an equivalent relation defined as 
$$
(x_1,...,x_{n+1})\sim (y_1,...,y_{n+1}) \Leftrightarrow \exists\; \lambda\in \CC^*\; \text{s.t.}\; x_i=\lambda^{a_i}y_i\; \text{for any $1\leq i\leq n+1$}.
$$

Assume that $M_*(\Gamma)$ is a free algebra and $f_i\in M_{k_i}(\Gamma)$, $1\leq i \leq n+1$, are free generators. Then the Satake-Baily-Borel compactification $(\cD_n/\Gamma)^*=\Proj (M_*(\Gamma))$ is the weighted projective space $\PP(k_1,...,k_{n+1})$. Moreover, the manifold $\Spm(M_*(\Gamma))=(\mathcal{L}_n/\Gamma)^*$ is the affine space $\CC^{n+1}$ and thus has no singular points. 

The next two results will be used later.
\begin{theorem}[\cite{Got69}]
\label{th:Got}
Let $\Gamma$ be a discrete group of analytic automorphisms of a complex manifold $X$, and let $\pi: X\to X/\Gamma$ be the natural morphism. A point $\pi(x)$ is nonsingular if and only if the stabilizer $\Gamma_x$ of $x$ in $\Gamma$ is generated by reflections whose mirrors pass through $x$. 
\end{theorem}

\begin{theorem}[\cite{Arm68}]
\label{th:Arm}
Let $\Gamma$ be a discrete group of homeomorphisms of a path connected topological space $X$. If the quotient space $X/\Gamma$ is simply connected, then $\Gamma$ is generated by elements having fixed points in $X$. 
\end{theorem}

The following result is a special case of \cite[Proposition 8.3]{VP89}. We give it a short proof.
\begin{proposition}\label{prop:group}
If $M_*(\Gamma)$ is a free algebra, then $\Gamma$ is generated by reflections.
\end{proposition}

\begin{proof}
Assume that $M_*(\Gamma)$ is a free algebra. Then $(\mathcal{L}_n/\Gamma)^*$ is an affine space. Since $(\mathcal{L}_n/\Gamma)^*$ is obtained from $\mathcal{L}_n/ \Gamma$ by adding zero and finitely many one- and two- dimensional cones, $\mathcal{L}_n/ \Gamma$ is smooth and simply connected. By Theorem \ref{th:Arm}, $\Gamma$ is generated by elements having fixed points. We then conclude from Theorem \ref{th:Got} that $\Gamma$ is generated by reflections.
\end{proof}

For any $r\in M$ with $(r,r)=-2$, we have $\sigma_r\in \widetilde{\Orth}^+(M)$. Conversely, 
as a consequence of Theorem 1.1 and Corollary 1.2 in \cite{GHS09}, one obtains the following sufficient condition for $\widetilde{\Orth}^+(M)$ to be generated by reflections.

\begin{lemma}\label{lem:group}
Let $M$ be an even lattice of signature $(2,n)$ with $n\geq 3$. Assume that $M$ contains an isotropic plane, represents $-2$, and $\rank_3(M)\geq 5$, $\rank_2(M)\geq 6$, where $\rank_p (M)$ is the maximal rank of the sublattices $M_1$ in $M$ such that $\det(M_1)$ is coprime to $p$. Then $\widetilde{\Orth}^+(M)$ is generated by $\sigma_r$ for $r\in M$ with $(r,r)=-2$.
\end{lemma}

The following theorem is vital to classify free algebras of orthogonal modular forms. The assertion (2) was stated in \cite[Proposition 6]{Vin13} with a brief idea of the proof. The last two assertions in the particular case of signature $(2,3)$ were mentioned at the end of \cite{Vin13}. We here give these results a full proof.

\begin{theorem}\label{th:freeJacobian}
Assume that $M_*(\Gamma)$ is a free algebra. Let $f_i\in M_{k_i}(\Gamma)$, $1\leq i \leq n+1$, be the generators.
\begin{enumerate}
\item The Jacobian determinant $J=J(f_1,...,f_{n+1})$ is not identically zero and it is a cusp form of weight $n+\sum_{i=1}^{n+1}k_i$ for $\Gamma$ with the character $\det$.

\item The zero divisor of $J$ is the sum of all mirrors of reflections in $\Gamma$ with multiplicity $1$. In particular, $J$ is a reflective cusp form.

\item Let $\{ \pi_1,...,\pi_s \}$ be the representatives of the $\Gamma$-equivalence classes of the mirrors of reflections in $\Gamma$. Then there exist a unique modular form $J_i$ for $\Gamma$ such that $\div(J_i)=2\cdot \Gamma \pi_i$ for each $1\leq i\leq s$, and $J^2=\prod_{i=1}^s J_i$.

\item There exist polynomials $P$, $P_i$, $1\leq i \leq s$, in $n+1$ variables over $\CC$ such that $J^2=P(f_1,...,f_{n+1})$ and $J_i =P_i(f_1,...,f_{n+1})$. Thus $P=\prod_{i=1}^s P_i$ and these $P_i$ are irreducible. 
\end{enumerate}
\end{theorem}

\begin{proof}
\begin{itemize}
\item[(1)] It follows from Theorem \ref{th:Jacobian}.
\item[(2)] Let $c$ be a zero-dimensional cusp and $\cH_c$ be the associated tube domain. Let $Z=(z_1,...,z_n)$ be a coordinate of $\cH_c$. We view $f_i$ as modular forms on $\cH_c$. For any $\mathcal{Z}\in \mathcal{L}_n$, there exist unique elements $Z\in \cH_c$ and $z_0\in \CC^*$ such that $\mathcal{Z}=z_0\cdot \pr_c (Z)$. Thus the function 
$$
\tilde{f_i}: \mathcal{L}_n \to \CC, \quad \tilde{f_i}(\mathcal{Z}):=z_0^{-k_i} f_i(Z)
$$
is defined well as a modular form of weight $k_i$ on $\mathcal{L}_n$. We choose $(z_0,z_1,...,z_n)$ as a coordinate of $\mathcal{L}_n$. The usual Jacobian determinant $\tilde{J}$ of the $n+1$ functions $\tilde{f_i}$ with respect to $(z_0,z_1,...,z_n)$ is equal to $J$ up to a power of $z_0$. By Baily-Borel compactification, $(\cD_n/\Gamma)^*=\Proj (M_*(\Gamma))$ is the weighted projective space $\PP(k_1,...,k_{n+1})$ and the natural isomorphism is given by the map 
$$
Z\mapsto [f_1(Z),...,f_{n+1}(Z)].
$$
Besides, $\mathcal{L}_n/\Gamma$ is an open subset of $\Spm(M_*(\Gamma))-\{0\}=\CC^{n+1}-\{0\}$. Thus we have the following holomorphic application
$$
\pi: \mathcal{L}_n \to \mathcal{L}_n/\Gamma \hookrightarrow \CC^{n+1}-\{0\},
$$
which is explicitly given by $\mathcal{Z}\mapsto (\tilde{f}_1(\mathcal{Z}),..., \tilde{f}_{n+1}(\mathcal{Z}))$.  For $v\in \mathcal{L}_n$ satisfying $\Gamma_v=\{ 1\}$, since $\Gamma$ is acting properly discontinuously on $\mathcal{L}_n$, the map $\pi$ is biholomorphic around $v$. Thus $\tilde{J}(v)\neq 0$, which yields $J(v)\neq 0$. By Theorem \ref{th:Got},  $J$ vanishes only on mirrors of reflections in $\Gamma$.  We then conclude from Theorem \ref{th:Jacobian} $(4)$ that $J$ vanishes exactly on all mirrors of reflections in $\Gamma$.

We next prove that the multiplicities are all one. Let $v$ be a vector such that $\sigma_v \in \Gamma$. For a generic point $ x\in v^\perp $, the stabilizer $\Gamma_x$ is generated by $\sigma_v$ and thus has order $2$. We can choose coordinate $(x_0, x_1,..., x_n)$ around $x$ such that $v^\perp=\{x_0=0\}$ and the map $\pi$ locally at $x$ is like 
$$
(x_0,x_1,..., x_n) \mapsto (x_0^2,x_1,..., x_n).
$$
Then it is straightforward to see that $J$ vanishes with multiplicity one along $\{x_0=0\}$.

\item[(3)] Since $J^2\in M_*(\Gamma)$, there exists a polynomial $P$ in $n+1$ variables over $\CC$ such that $J^2=P(f_1,...,f_{n+1})$. On the one hand, in $\mathcal{L}_n$ we have $\div(J^2)=2\sum_{i=1}^s \Gamma \pi_i$ and the divisor of $J^2$ in $\mathcal{L}_n/ \Gamma \hookrightarrow \CC^{n+1}-\{0\}$ is the sum of hyperplanes $\pi_i$. In the other hand, suppose $P=P_1\cdots P_t$ is the irreducible decomposition over $\CC$. Then we have the irreducible decomposition of zero locus in $\PP(k_1,...,k_{n+1})$: $Z(P)=Z(P_1)\cup \cdots \cup Z(P_t)$. Thus each $Z(P_i)$ will correspond to a $\pi_j$. By comparing the order of divisor,  the desired claims are proved.
\end{itemize}
\end{proof}

\begin{remark}
If $M=U\oplus U(m)\oplus L(-1)$ and $\widetilde{\Orth}^+(M)<\Gamma$, then each $J_i$ will be a Borcherds product by \cite{Bru14}. We denote the input by $\phi_i$. Thus the Borcherds product $F_i$ of $\frac{1}{2} \phi_i$ will give a modular form whose divisor is $\Gamma \pi_i$. It is clear that $F_i^2=J_i$. Therefore each $F_i$ is a modular form for $\Gamma$ with a character (or multiplier system) of order $2$ and we have $J=\prod_{i=1}^s F_i$.
\end{remark}

\section{Classification of free algebras of orthogonal modular forms}\label{sec:classification}

Let $M=U\oplus U(m)\oplus L(-1)$, where $m$ is a positive integer. Assume that $\widetilde{\Orth}^+(M)<\Gamma<\Orth^+(M)$. In this section we classify the groups $\Gamma$ such that $M_*(\Gamma)$ is a free algebra.

It was proved in \cite{SV17} that there is no arithmetic group $\Gamma$ such that $M_*(\Gamma)$ is a free algebra when the signature $(2,n)$ satisfies $n> 10$. We here prove a special case of this result. To this aim, we need the following lemma.

\begin{lemma}\label{lem:complete 2-divisor}
If there is a modular form with complete $2$-divisor for $\widetilde{\Orth}^+(M)$, then there is also a modular form with complete $2$-divisor for $\widetilde{\Orth}^+(M_1)$, where $M_1$ is an even overlattice of $M$.
\end{lemma}

\begin{proof}
Let $F$ be a modular form of weight $k$ with complete $2$-divisor for $\widetilde{\Orth}^+(M)$. By \cite{Bru14}, it is a Borcherds product of a nearly holomorphic modular form of weight $-\rank(L)/2$ for the Weil representation of $\SL_2(\ZZ)$ attached to the discriminant group of $M$. We denote this input by $f$. The principal part of $f$ is $(q^{-1}+2k)\textbf{e}_0$. By \cite[Lemma 5.6]{Bru02}, the lifting $f\vert\uparrow_M^{M_1}$ gives a vector-valued modular form of the same weight for the Weil representation attached to $D(M_1)$, and this modular form has the principal part $(q^{-1}+2k')\textbf{e}_0$, where $k'$ is a positive integer. Thus the Borcherds product of $f\vert \uparrow_M^{M_1}$ gives a modular form with complete 2-divisor for $\widetilde{\Orth}^+(M_1)$. This completes the proof.
\end{proof}

\begin{remark}
When $M$ is not of the form $U\oplus U(m)\oplus L(-1)$, the above lemma is not true. For example, $2U(2)\oplus 5A_1(-1)$ has a modular form with complete 2-divisor (see \cite[\S 6.2]{GN18}), but $2U\oplus 5A_1(-1)$ has no modular forms with complete 2-divisor (see \cite[Theorem 6.9]{Wan19}).
\end{remark}

\begin{theorem}\label{th:nonfree}
Let $M=U\oplus U(m)\oplus L(-1)$ and $\widetilde{\Orth}^+(M)<\Gamma <\Orth^+(M)$. If the graded algebra $M_*(\Gamma)$ is free, then $\rank(L)\leq 8$.
\end{theorem}

\begin{proof}
Assume that $M_*(\Gamma)$ is a free algebra. Then the Jacobian determinant of generators gives a reflective modular form. Since $\widetilde{\Orth}^+(M)<\Gamma$, all reflections $\sigma_r$ with $(r,r)=-2$ belong to $\Gamma$. By Theorem \ref{th:freeJacobian}, the decomposition of the Jacobian determinant will give a modular form with complete 2-divisor. We know from \cite[Theorem 3.4]{Wan18} that if there exists a modular form with complete 2-divisor for $M = 2U\oplus L(-1)$ then either $\rank(L) \leq 8$ or $L$ is a unimodular lattice of rank 16 or 24. By Lemma \ref{lem:complete 2-divisor}, we only need to consider the two cases $U\oplus U(m)\oplus 2E_8$ and $U\oplus U(m)\oplus 3E_8$. When $m>1$, the two lattices are isomorphic to some lattices containing $2U$ by \cite{Nik80} because the minimal number of generators of the associated discriminant groups is 2. It follows that there is no modular form with complete 2-divisor for these lattices. When $m=1$, the orthogonal groups of $2U\oplus 2E_8$ and $2U\oplus 3E_8$ contain only 2-reflections. The corresponding 2-reflective modular forms have weights $132$ and $12$. The weight is too small, which leads to a contradiction. For example, in the case of $2U\oplus 2E_8$, the singular weight is 8. Since $132< 8\times 19 + 18$, the unique reflective modular form is impossible to be the Jacobian of free generators. Thus the space of modular forms is not free in the case of $U\oplus U(m)\oplus 2E_8$. The proof is completed.
\end{proof}

We next classify free algebras of modular forms in the case of $M=2U\oplus L(-1)$. The following classification result is our main theorem.

\begin{theorem}
\label{th:classification}
Let $M=2U\oplus L(-1)$ and $\widetilde{\Orth}^+(M)<\Gamma<\Orth^+(M)$. If $M_*(\Gamma)$ is a free algebra, then $(L,\Gamma)$ can only take one of the following $26$ pairs
\begin{align*}
&(A_1, \Orth^+)& &(2A_1, \Orth^+)& &(3A_1, \Orth^+)& &(4A_1,\Orth^+)& &(A_2, \widetilde{\Orth}^+)& &(A_2, \Orth^+)& \\
&(A_3, \widetilde{\Orth}^+)& &(A_3, \Orth^+)&  &(A_4, \widetilde{\Orth}^+)& &(A_5, \widetilde{\Orth}^+)& &(A_6, \widetilde{\Orth}^+)& &(A_7, \widetilde{\Orth}^+)&  \\
&(D_4, \widetilde{\Orth}^+)& &(D_5, \widetilde{\Orth}^+)& &(D_6, \widetilde{\Orth}^+)& &(D_7, \widetilde{\Orth}^+)& &(D_8, \widetilde{\Orth}^+)&\\
&(D_4, \Orth^+)& &(D_5, \Orth^+)& &(D_6, \Orth^+)& &(D_7, \Orth^+)& &(D_8, \Orth^+)&\\
&(D_4, \Orth_1^+)& &(E_6, \widetilde{\Orth}^+)& &(E_7, \Orth^+)& &(E_8, \Orth^+)&
\end{align*}
where $\Orth^+$ stands for the full orthogonal group, $\widetilde{\Orth}^+$ denotes the discriminant kernel, and $\Orth_1^+$ is the subgroup generated by $\widetilde{\Orth}^+$ and a sign change of odd number of coordinates in $D_4\otimes\CC$. 
\end{theorem}

The following lemma is an advanced version of the Jacobi forms approach used to classify reflective modular forms in \cite{Wan19}.

\begin{lemma}
\label{lem:main}
Let $M=2U\oplus L(-1)$ and $\widetilde{\Orth}^+(M)<\Gamma<\Orth^+(M)$. Assume that $F$ is a reflective modular form of weight $k$ for $\Gamma$ whose zero divisor is the sum of all mirrors of reflections in $\Gamma$ with multiplicity one. We define
$$
R(L^\vee)=\{ x\in L^\vee : \text{$F$ vanishes on $\cD_{(0,0,x,1,0)}$}  \}.
$$
Let $R(L)$ be the subset of $L$ consisting of vectors $d_x x$, where $x\in R(L^\vee)$ and $d_x$ is the order of $x$ in $L^\vee/L$.
Suppose that $R(L)$ is non-empty.  Then the set $R(L)$ defines a reduced root system of rank equal to $\rank(L)$. Moreover, $R(L)$ is a direct sum of some irreducible root systems and all  irreducible components have the same modified Coxeter numbers defined below (a given root system may have different modified Coxeter numbers in our definition; the reason is given in the proof).
\begin{enumerate}
\item $A_n(d)$ with $n\geq 1$ and $d\geq 1$. modified Coxeter number: $(n+1)/d$.
\item $B_n(2d)$ with $n\geq 2$ and $d\geq 1$. modified Coxeter number: $(n+1)/d$.
\item $C_n(d)$ with $n\geq 3$ and $d\geq 1$. modified Coxeter number: $(2n-1)/d$.
\item $D_n(d)$ with $n\geq 4$ and $d\geq 1$. modified Coxeter number: $2(n-1)/d$.
\item $E_6(d)$ with $d\geq 1$. modified Coxeter number: $12/d$. 
\item $E_7(d)$ with $d\geq 1$. modified Coxeter number: $18/d$.
\item $E_8(d)$ with $d\geq 1$. modified Coxeter number: $30/d$.
\item $G_2(d)$ with $d\geq 1$. modified Coxeter number: $4/d$.
\item $F_4(2d)$ with $d\geq 1$. modified Coxeter number: $9/d$.
\item $A_1(d)$ with $d\geq 1$. modified Coxeter number: $1/(2d)$.
\item $A_1(d)$ with $d\geq 1$. modified Coxeter number: $3/(2d)$.
\item $B_n(2d)$ with $n\geq 2$ and $d\geq 1$. modified Coxeter number: $(2n-1)/(2d)$.
\item $B_n(2d)$ with $n\geq 2$ and $d\geq 1$. modified Coxeter number: $(2n+1)/(2d)$.
\end{enumerate} 
\end{lemma}

\begin{proof}
By \cite{Bru14}, $F$ should be a Borcherds product. In view of the isomorphism between the spaces of vector-valued modular forms and Jacobi forms, there exists a weakly holomorphic Jacobi form $\phi$ of weight $0$ and index $1$ for $L$ such that $F=\Borch(\phi)$ (see Theorem \ref{th:Borcherds}). The divisors of the form $\cD_{(0,0,x,1,0)}$ determine the $q^0$-term of $\phi$. More precisely, we have
$$
\phi=q^{-1} + \sum_{\substack{x \in R(L^\vee)\\ 2x\not\in R(L^\vee)\\ x/2 \not\in L^\vee}} \zeta^x + \sum_{\substack{y \in R(L^\vee)\\ y/2 \in L^\vee}} (\zeta^y - f(y)\zeta^{y/2}) +2k +O(q),
$$
where $f(y)=0$ if $y/2 \in R(L^\vee)$ and $f(y)=1$ if $y/2 \not\in R(L^\vee)$, because the above $q^0$-term of the input determines that  in the zero divisor of $\Borch(\phi)$, the divisor $\cD_{(0,0,y,1,0)}$ has multiplicity $1$ and the divisor $\cD_{(0,0,y/2,1,0)}$ has multiplicity $1-f(y)$ (see Theorem \ref{th:Borcherds}). Remark that the term $q^{-1}$ corresponds to the divisor $\cD_{(0,-1,0,1,0)}$. For convenience, we write $\phi=q^{-1}+ \sum f(0,r)\zeta^r + O(q)$. By \cite[Proposition 2.6]{Gri18}, we have
\begin{align}
\label{eq:4.1}\sum_{r\in L^\vee} f(0,r) (r,\mathfrak{z})^2= 2C(\mathfrak{z},\mathfrak{z}), \quad \mathfrak{z}\in L\otimes \CC,\\
\label{eq:4.2}C=\frac{1}{24}\sum_{r\in L^\vee} f(0,r) -1 = \frac{1}{2\rank(L)} \sum_{r\in L^\vee} f(0,r) (r,r).
\end{align}
We remark that the above constant $C$ also appears in the Weyl vector $(A, \vec{B}, C)$ of the Borcherds product $F$ and it satisfies the relation $A=C+1$ in this case.
We claim that $R(L)$ generates the whole space $L\otimes \RR$, otherwise there will be a vector in $L\otimes \CC$ orthogonal to $R(L)$, which contradicts the first identity. By definition, we have that $\sigma_{(0,0,u,1,0)}\in \Gamma$ for $u\in R(L^\vee)$. Let $v\in R(L^\vee)$.  Since
$$
\sigma_{(0,0,u,1,0)}((0,0,v,1,0))= (0,0,\sigma_u(v),1-2(u,v)/(u,u),0)=:l,
$$
we have $\sigma_u(v)\in L^\vee$ and $2(u,v)/(u,u) \in \ZZ$ for any $v\in R(L^\vee)$. Moreover, $F$ vanishes on the hyperplane $l^\perp$. Notice that $(0,0,\sigma_u(v),1,0)$ is primitive in $L^\vee$. By the Eichler criterion (see \cite[Proposition 3.3]{GHS09}), there exists $g\in \widetilde{\Orth}^+(M)$ such that $g(l)=(0,0,\sigma_u(v),1,0)$. Thus $F$ also vanishes on $(0,0,\sigma_u(v),1,0)^\perp$, which yields $\sigma_u(v)\in R(L^\vee)$. It follows that $R(L)$ is a reduced root system. Therefore $R(L)$ can be written as a direct sum of rescaled irreducible root systems (see \cite{Bou60}). Let $R$ be an irreducible component of $R(L)$ and $R^*$ be the corresponding component in $R(L^\vee)$.  The modified Coxeter number of $R$ will be defined as the constant $C$ in the identity of type \eqref{eq:4.1} for $R^*$.  Recall that the usual Coxeter number of an irreducible root system $\mathcal{R}$ is defined as the constant $h$ appearing in the following equality
$$
\sum_{r\in \mathcal{R}} (r,\mathfrak{z})^2 = 2h (\mathfrak{z},\mathfrak{z}), \quad \mathfrak{z}\in \mathcal{R}\otimes \CC. 
$$
\begin{itemize}
\item[(a)] If $R$ equals $A_n(d)$ with $n\geq 2$, $D_n(d)$ with $n\geq 4$, $E_6(d)$, $E_7(d)$, or $E_8(d)$, every root $r\in R$ has norm $2d$ and $\div(r)=d$ in $R$, here we view $R$ as a lattice generated by its roots. Thus $\div(r)=d$ in $L$ because it defines a reflective vector in $M$. In this case, $R^*=A_n(\frac{1}{d})$, $D_n(\frac{1}{d})$, $E_6(\frac{1}{d})$, $E_7(\frac{1}{d})$, $E_8(\frac{1}{d})$, respectively. By \eqref{eq:4.1}, it is easy to prove that the constant $C$ is equal to $h/d$, where $h$ is the usual Coxeter number of $R$.

\item[(b)] If $R=C_n(d)$ with $n\geq 3$, every short root $r$ has norm $2d$ and $\div(r)=d$ in $R$. Thus $\div(r)=d$ in $L$. Every long root $s$ has norm $4d$ and $\div(s)=2d$ in $R$. Thus $\div(s)=2d$ in $L$. In this case, $R^*=B_n(\frac{1}{d})$. By \eqref{eq:4.1}, $C=h/d$, where $h=2n-1$ is the Coxeter number of $B_n$.

\item[(c)] If $R=G_2(d)$, every short root $r$ has norm $2d$ and $\div(r)=d$ in $R$. Thus $\div(r)=d$ in $L$. Every long root $s$ has norm $6d$ and $\div(s)=3d$ in $R$. Thus $\div(s)=3d$ in $L$. In this case, $R^*=G_2(\frac{1}{3d})$. By \eqref{eq:4.1}, $C=4/d$.

\item[(d)] If $R=F_4(2d)$, every short root $r$ has norm $2d$ and $\div(r)=d$ in $R$. Thus $\div(r)=d$ in $L$. Every long root $s$ has norm $4d$ and $\div(s)=2d$ in $R$. Thus $\div(s)=2d$ in $L$. In this case, $R^*=F_4(\frac{1}{d})$. By \eqref{eq:4.1}, $C=9/d$.

\item[(e)] If $R=A_1(d)$, every root $r$ has norm $2d$ and $\div(r)=2d$ in $R$. If $\div(r)=d$ in $L$, then $R^*=A_1(\frac{1}{d})$ and $C=\frac{2}{d}$. If $\div(r)=2d$ in $L$, then there are three possible cases:
\begin{itemize}
\item[(i)] $r/d \not\in R(L^\vee)$. In this case, $R^*=A_1(\frac{1}{4d})$ and $C=\frac{1}{2d}$. 
\item[(ii)] $r/d \in R(L^\vee)$ and $r/(2d)\in R(L^\vee)$. In this case, $R^*=A_1(\frac{1}{d})\cup A_1(\frac{1}{4d})$ and $C=\frac{2}{d}$.
\item[(iii)] $r/d \in R(L^\vee)$ and $r/(2d)\not\in R(L^\vee)$. In this case, $R^*=A_1(\frac{1}{d})$ but $C=\frac{3}{2d}$.
\end{itemize}

\item[(f)] If $R=B_n(2d)$ with $n\geq 2$, every short root $r$ has norm $2d$ and $\div(r)=2d$ in $R$. Every long root $s$ has norm $4d$ and $\div(s)=2d$. Thus $\div(s)=2d$ in $L$. If $\div(r)=d$ in $L$, then $R^*=C_n(\frac{1}{2d})$ and $C=\frac{n+1}{d}$. If $\div(r)=2d$ in $L$, then there are three possible cases:
\begin{itemize}
\item[(i)] $r/d \not\in R(L^\vee)$. In this case, $R^*=B_n(\frac{1}{2d})$ and $C=\frac{2n-1}{2d}$. 
\item[(ii)] $r/d \in R(L^\vee)$ and $r/(2d)\in R(L^\vee)$. In this case, $R^*=C_n(\frac{1}{2d})\cup nA_1(\frac{1}{4d})$ and $C=\frac{n+1}{d}$.
\item[(iii)] $r/d \in R(L^\vee)$ and $r/(2d)\not\in R(L^\vee)$. In this case, $R^*=C_n(\frac{1}{2d})$ but $C=\frac{2n+1}{2d}$.
\end{itemize}

\end{itemize}
By the above discussions, we complete the proof.
\end{proof}

\begin{proof}[Proof of Theorem \ref{th:classification}]
Suppose that $M_*(\Gamma)$ is a free algebra.  Then the Jacobian determinant $J$ of $\rank(L)+3$ free generators is a cusp form and it defines a reflective modular form satisfying all conditions in Lemma \ref{lem:main}. The set $R(L)$ is non-empty, otherwise the reflective modular form has weight 12 and is not a cusp form because it has the Weyl vector $(1,0,0)$.  We only need to consider the case of $\rank(L)\leq 8$. The modified Coxeter number is just the constant $C$ in the Weyl vector of the Borcherds product (see the above lemma and Theorem \ref{th:Borcherds}). The term $q^A \zeta^{\vec{B}}\xi^C$ corresponding to the Weyl vector is one of the first Fourier coefficients of the Jacobian determinant. From the assertion $(5)$ of Theorem \ref{th:Jacobian},  we derive that $C\geq (\rank(L)+2)-1$ and thus the number $C$ is an integer no less than $\rank(L)+1$. This forces that the irreducible components of $R(L)$ must be of type (1)--(5), (8), (9) with $d=1$, or type  (6) with $d=1,2$, or type (7) with $d=1,2,3$. All possible cases are as follows:
\begin{itemize}
\item[(a)] When $R(L)=A_n$, $L$ must be $A_n$ because it is an even overlattice of $R(L)$. By \cite{Wan19}, $2U\oplus A_8(-1)$ has no modular forms with complete 2-divisor, which yields that $M_*(\Gamma)$ is not a free algebra in this case. By Theorem \ref{th:freeJacobian} (3), if $\Gamma$ contains a $2d$-reflection  then there is a modular form vanishing exactly on the $\Gamma$-orbit of the mirror of this reflection. This modular form should be a Borcherds product of a Jacobi form. From the $q^0$-term of this Jacobi form, we see that  there is a $2d$-reflection in $R(L)$. Thus $\Gamma$ must be the discriminant kernel because it only contains 2-reflections. Note that when $n=1$ we have $\widetilde{\Orth}^+=\Orth^+$.

\item[(b)] When $R(L)=D_n$ with $4\leq n \leq 8$, $L$ must be $D_n$.  Similarly, $\Gamma = \widetilde{\Orth}^+$. 

\item[(c)] When $R(L)=C_n$ with $3\leq n \leq 8$, $L$ is equal to $D_n$. Note that $A_3=D_3$. The Weyl group of $B_n$ is equal to $\Orth(D_n)$ if $n\neq 4$ and is generated by $W(D_4)$ and the odd sign change if $n=4$. It is easy to check that the natural homomorphism $\Orth(D_n)\to \Orth(D(D_n))$ is surjective. Thus $\Orth^+$ is generated by $\widetilde{\Orth}^+$ and $\Orth(D_n)$.  Hence $\Gamma = \Orth^+$ if $n\neq 4$ and $\Gamma = \Orth_1^+$ if $n=4$. 

\item[(d)] When $R(L)=B_n(2)$ with $2\leq n \leq 8$, $L$ is equal to $nA_1$ or $N_8$, where $N_8\cong D_8^\vee (2)$ is the Nikulin lattice whose root sublattice is $8A_1$. Note that $\Orth(nA_1)=W(C_n)$ and the natural homomorphism $\Orth(nA_1)\to \Orth(D(nA_1))$ is surjective. Thus $\Gamma=\Orth^+$ when $2\leq n\leq 4$. It is known by \cite{Wan19} that there is no modular form with complete 2-divisor for $2U\oplus nA_1(-1)$ when $n\geq 5$. It remains to consider the case of $L=N_8$. In this case, we have $W(B_8)=W(C_8)=\Orth(D_8)=\Orth(D_8^\vee)=\Orth(N_8)$, which implies that $\Gamma=\Orth^+$. Since $N_8$ has level $2$, there are only $2$-reflections and $4$-reflections. The $2$-reflective modular form and $4$-reflective modular form have the same weight $28$ (the two modular forms do exist; we refer to \cite{GN18} for a construction.). The Weyl vector of the Jacobian determinant has the form $(10,*,9)$. Suppose that the algebra of orthogonal modular forms in the case of $N_8$ is free. Then there will be ten generators of weight 4 and one generator of weight 6, because the sum of the weights of the eleven generators is equal to $46=56-10$. We can kill the first Fourier-Jacobi coefficients of a given modular form of weight 4 by a linear combination of the generators of weight 4. Therefore there will be a generator of weight 4 with Fourier expansion of the form $q^2\xi^2(\cdots)$. This forces that the Fourier expansion of the Jacobian determinant of generators has the form $q^{10}\xi^{10}(\cdots)$, which contradicts the Weyl vector of the Jacobian determinant.

\item[(e)] When $R(L)=G_2$, $L=A_2$ and $W(G_2)=\Orth(A_2)$. Thus $\Gamma=\Orth^+$.

\item[(f)] When $R(L)=F_4(2)$, $L=D_4$ and $W(F_4)=\Orth(D_4)$. Thus $\Gamma=\Orth^+$.

\item[(g)] When $R(L)=E_6$, $L=E_6$ and $\Gamma=\widetilde{\Orth}^+$.

\item[(h)] When $R(L)=E_7$, $L=E_7$ and $\Gamma=\widetilde{\Orth}^+=\Orth^+$.

\item[(i)] When $R(L)=E_8$, $L=E_8$ and $\Gamma=\widetilde{\Orth}^+=\Orth^+$.

\item[(j)] When $R(L)=E_8(3)$, by \eqref{eq:4.2}, the weight of $J$ will be given by
$$
\frac{1}{24}(240+2k)-1=\frac{30}{3},
$$
which follows that $k=12$. This is impossible.

\item[(k)] When $R(L)=E_8(2)$, $L$ is of level 2 and equal to $E_8(2)$, otherwise $L$ will contain 2-roots which contradicts the assumption $R(L)=E_8(2)$. Since the natural homomorphism $\Orth(E_8)\to \Orth(D(E_8(2)))$ is surjective, $\Orth^+$ is generated by $\widetilde{\Orth}^+$ and $\Orth(E_8)$. It follows that $\Gamma=\Orth^+$. Note that the $2$-reflective modular form and $4$-reflective modular form for $\Gamma$ have weights $12$ and $60$, respectively.
It is known from \cite[\S 6]{Wan18b} that $\dim M_4(\Gamma)=\dim M_6(\Gamma)=1$. This contradicts the weight of the Jacobian determinant because $12+60<10+4+6+8\times 9$. This case also follows from the case (m) below because we have the following isomorphisms among orthogonal groups:
\begin{align*}
\Orth^+(2U\oplus E_8(2))&=\Orth^+((2U\oplus E_8(2))^\vee)=\Orth^+(2U\oplus E_8(1/2))\\
&=\Orth^+(2U(2)\oplus E_8)\cong\Orth^+(2U\oplus 2D_4).
\end{align*}

\item[(l)] When $R(L)=E_7(2)$, we have $E_7(2)<L$ and $E_7(\frac{1}{2})<L^\vee$, which forces that $L=E_7(2)$.
We deduce from \eqref{eq:4.2} that $J$ has weight $57$ and the Weyl vector $(10,*,9)$. Since $\Orth(E_7)=W(E_7)$ is contained in $\Gamma$, the non-zero modular forms for $\Gamma$ have even weight. It is easy to prove that $\dim M_4(\Gamma) \leq 3$ using the argument in \cite[\S 3]{WW20}. Indeed, we take $\delta=10$ in \cite[Inequality (3.3) in \S 3]{WW20}. Notice that $\Orth(E_7(2))=\Orth(E_7)=W(E_7)$ and $J_{k,E_7(2),m}^{w,W(E_7)}=J_{k,E_7,2m}^{w,W(E_7)}$.  We then obtain the upper bound of the dimension by $\dim J_{4,E_7,0}^{w,W(E_7)}=\dim J_{-8,E_7,2}^{w,W(E_7)}=\dim J_{-20,E_7,4}^{w,W(E_7)}=1$.  Since $57-9< 4\times 3 + 6 \times 7$, we get a contradiction.

\item[(m)] When $R(L)=2F_4(2)$, we have $L>2D_4$. Thus $L$ can only take $2D_4$, $D_8$ or $E_8$. We only need to consider the case $L=2D_4$. It is easy to see that the exchange of two copies of $D_4$ does not belong to $\Gamma$. If $M_*(\Gamma)$ is a free algebra, we know by Theorem \ref{th:freeJacobian} that the decomposition of the Jacobian determinant $J$ will give a modular form with divisor $\Gamma v^\perp$, where $v$ is a $4$-reflective vector in the first copy of $D_4$. This modular form should be a Borcherds product of a weak Jacobi form of weight 0 and index 1 for $2D_4$. By the similar argument as in the proof of Lemma \ref{lem:main}, this leads to a contradiction because the $4$-reflective vectors in the first copy of $D_4$ do not span the whole space of dimension 8.
\end{itemize}
We then finish the proof of the theorem.
\end{proof}

\section{A sufficient condition to be free algebras}\label{sec:sufficient}
In this section we prove that the converse of Theorem \ref{th:freeJacobian} holds, which gives a sufficient condition for the graded algebra of orthogonal modular forms being free.

\begin{theorem}\label{th:converseJacobian}
Let $\Gamma<\Orth^+_{2,n}$ be an arithmetic group. If there exists a modular form $F$ (with a character) on $\Gamma$  which vanishes exactly on all mirrors of reflections in $\Gamma$ with multiplicity one and equals the Jacobian determinant of $n+1$ certain modular forms on $\Gamma$, then the graded algebra $M_*(\Gamma)$ is freely generated by the $n+1$ modular forms. Moreover, the group $\Gamma$ is generated by all reflections whose mirrors are contained in the divisor of the modular form $F$.
\end{theorem}

\begin{proof}
Assume that $f_i \in M_{k_i}(\Gamma)$, $1\leq i\leq n+1$, and $F=J=J(f_1,...,f_{n+1})$ vanishes exactly on all mirrors of reflections in $\Gamma$.  Suppose that $M_*(\Gamma)$ is not a free algebra.  Then there are non-trivial modular forms not in $\CC[f_1, ...,f_{n+1}]$. Let $f_{n+2}\in M_{k_{n+2}}(\Gamma)$ be such a modular form of minimal weight.  For $1\leq t \leq n+2$ we define $J_t$ as the Jacobian determinant of the $n+1$ modular forms $f_i$ except $f_t$. It is clear that $J=J_{n+2}$.  By Theorem \ref{th:Jacobian} (4), the quotient $J_t/J$ is a holomorphic modular form on $\Gamma$ and we denote it by $g_t$. It is easy to check that the following identity holds:
$$
\sum_{t=1}^{n+2} (-1)^t k_t f_t J_t = 0.
$$
By $J_t=Jg_t$, we have 
$$
\sum_{t=1}^{n+2} (-1)^t k_t f_t g_t = 0,
$$
which yields 
$$
(-1)^{n+2}k_{n+2}f_{n+2}=-\sum_{t=1}^{n+1}(-1)^t k_t f_t g_t
$$
because $g_{n+2}=1$. The assumption on the weight of $f_{n+2}$ forces that all $g_t$ are contained in $\CC[f_1, ...,f_{n+1}]$.  Then $f_{n+2}\in \CC[f_1, ...,f_{n+1}]$, which leads to a contradiction. Hence the graded algebra $M_*(\Gamma)$ is free. From Proposition \ref{prop:group} and Theorem \ref{th:freeJacobian} (2), we conclude that $\Gamma$ is generated by all reflections related to the divisor of the Jacobian determinant. The proof is completed.
\end{proof}

For the 26 orthogonal groups in Theorem \ref{th:classification}, the expected Jacobian determinant of generators can be constructed as quasi pull-backs of the Borcherds form of weight 12 for $\II_{2,26}$ (see \cite{GN18}). 
Thus it is possible to prove the associated algebras of orthogonal modular forms are free using the above theorem.  The main difficulty is to verify that the Jacobian is not zero, or equivalently, those generators are algebraically independent. 
But the computation will be very cumbersome when the dimension of the modular variety is large, especially in the case of $E_8$.

We give an application of our result. The main theorem in \cite{FS07} asserts that the graded algebra of modular forms on $\widetilde{\Orth}^+(2U(2)\oplus D_4(-1))$ is freely generated by six forms of weight 2 and one form of weight 6. It was also verified that the Jacobian determinant of the seven modular forms is not zero. We know from \cite{Woi17} that there is a reflective modular form of weight 24 on $\Orth^+(2U\oplus D_4(-1))$ whose divisor is a sum of $\cD_v$ for all $v\in 2U\oplus D_4(-1)$ with $(v,v)=-4$ and $\div(v)=2$. In view of the isomorphisms
$$
\Orth^+(2U\oplus D_4(-1))=\Orth^+(2U\oplus D_4^\vee(-1))=\Orth^+(2U(2)\oplus D_4(-1)), \quad D_4^\vee(2)\cong D_4,
$$
this reflective modular form can be regarded as a 2-reflective modular form on $\Orth^+(2U(2)\oplus D_4(-1))$. This gives a new proof of the main theorem in \cite{FS07}.  This also proves that $\widetilde{\Orth}^+(2U(2)\oplus D_4(-1))$ is generated by all 2-reflections, which can not be covered by Lemma \ref{lem:group}. The structure results in \cite{AI05} can also be verified in a similar way.

It is clear that Theorem \ref{th:converseJacobian} also holds for Hilbert modular forms with respect to real quadratic fields because the Koecher principle is satisfied in this case. The first free algebras of Hilbert modular forms was determined by Gundlach \cite{Gun63}. He showed that the space of symmetric Hilbert modular forms of even weight for $\SL_2(\mathcal{O}_F)$ with $F=\QQ(\sqrt{5})$  is freely generated by three forms of weights 2, 6, 10. This space can be identified with the algebra of modular forms on $\Orth^+(U\oplus \mathbb{B}_5)$, where $\mathbb{B}_5=\left(\begin{smallmatrix}
2 & 1 \\ 
1 & -2
\end{smallmatrix}\right) $. The Jacobian determinant of the three generators is the product of two Gundlach's cusp forms of weights 5 and 15, which can be constructed as a reflective Borcherds product. Thus we can recover Gundlach's theorem.

We hope to construct more free algebras of orthogonal modular forms using the above theorem.

At the end of the paper,  we formulate the following conjecture, which also gives a nice way to construct free algebras of modular forms.

\begin{conjecture}
Let $\Gamma<\Orth^+_{2,n}$ be an arithmetic group generated by reflections. Let $\Gamma'$ be a finite index subgroup of $\Gamma$. If $M_*(\Gamma')$ is a free algebra, then the smaller algebra $M_*(\Gamma)$ is also free.
\end{conjecture}

By Lemma \ref{lem:group}, $\Orth^+(2U\oplus E_8(-1))$ is generated by reflections. Since $D_8$ is a sublattice of $E_8$, $\widetilde{\Orth}^+(2U\oplus D_8(-1))$ is a finite index subgroup of $\Orth^+(2U\oplus E_8(-1))$. We know from \cite{WW20}  that $M_*(\widetilde{\Orth}^+(2U\oplus D_8(-1)))$ is a free algebra. Thus the above conjecture implies the freeness of the algebra $M_*(\Orth^+(2U\oplus E_8(-1)))$. Similarly, the freeness of $M_*(\widetilde{\Orth}^+(2U\oplus A_7(-1)))$ implies the freeness of $M_*(\widetilde{\Orth}^+(2U\oplus E_7(-1)))$. But $M_*(\Orth^+(2U\oplus A_7(-1)))$ is not free because $\Orth^+(2U\oplus A_7(-1))$ is not generated by reflections.

We remark that the modularity of formal Fourier-Jacobi expansions of modular forms on $\Orth^+(2U\oplus E_8(-1))$ holds (see \cite[Corollary 4.4]{WW20} for the definition). In fact, every formal Fourier-Jacobi expansion for $E_8$ is automatically a formal Fourier-Jacobi expansion for $D_8$. Therefore, the modularity in the case of $E_8$ follows from the modularity in the case of $D_8$ proved in \cite{WW20}.

\bigskip

\noindent
\textbf{Acknowledgements} 
I would like to thank Eberhard Freitag for many valuable comments,  Ernest Vinberg for answering questions related to the paper \cite{Vin13}, and Zhiwei Zheng for many fruitful discussions.  I am greatly indebted to Riccardo Salvati Manni for suggesting the proof of Theorem \ref{th:converseJacobian} and for many helpful discussions. I also like to thank Brandon Williams for performing many computer calculations. I am grateful to Max Planck Institute for Mathematics in Bonn for its hospitality and financial support. I also thank the referee for his/her careful reading and useful comments.

\bibliographystyle{amsalpha}

\begin{thebibliography}{Wan18b}

\bibitem[AI05]{AI05} H. Aoki, T. Ibukiyama,\textit{Simple graded rings of Siegel modular forms, differential operators and Borcherds products.} Intern. J. Math. \textbf{16}:3 (2005), 249--279.

\bibitem[Arm68]{Arm68} M. A. Armstrong, \textit{The fundamental group of the orbit space of a discontinuous group.} Math. Proc. Camb. Philos. Soc. \textbf{64}:2 (1968), 299--301.

\bibitem[BB66]{BB66} W. L. Baily, A. Borel,\textit{Compactification of arithmetic quotients of bounded symmetric domains.} Ann. of Math. (2), \textbf{84} (1966), 442--528.

\bibitem[Bor00]{Bor00} R. E. Borcherds, \textit{Reflection groups of Lorentzian lattices. } Duke Math. J. \textbf{104} (2000), no.2, 319--366.

\bibitem[Bor95]{Bor95} R. E. Borcherds, \textit{Automorphic forms on $\Orth_{s+2,2}$ and infinite products.} Invent. Math. \textbf{120} (1995), no. 1, 161--213.

\bibitem[Bor98]{Bor98} R. E. Borcherds, \textit{Automorphic forms with singularities on Grassmannians.} Invent. Math. \textbf{123} (1998), no. 3, 491--562.

\bibitem[Bou60]{Bou60} N. Bourbaki, {\it Groupes et alg\`{e}bres de Lie.} Chapter 4,5 et 6.

\bibitem[Bru02]{Bru02} J. H. Bruinier, \textit{Borcherds products on $\Orth (2,n)$ and Chern classes of Heegner divisors.} Lecture Notes in Mathematics, vol. \textbf{1780}. Springer-Verlag, Berlin, 2002.

\bibitem[Bru14]{Bru14} J. H. Bruinier, \textit{On the converse theorem for Borcherds products.} J. Algebra \textbf{397} (2014), 315--342.

\bibitem[CG13]{CG13} F. Cl\'ery,   V.~Gritsenko,
\textit{Modular forms of orthogonal type and Jacobi theta-series.}
Abh. Math. Semin. Univ. Hambg. \textbf{83} (2013), 187--217.

\bibitem[Che55]{Che55} C. Chevalley, \textit{Invariants of finite groups generated by reflections.} Amer. J. Math. \textbf{77} (1955), 778--782.

\bibitem[Dit19]{Dit19} M. Dittmann, \textit{Reflective automorphic forms on lattices of squarefree level.}  Trans. Amer. Math. Soc. \textbf{372} (2019), 1333--1362.

\bibitem[DK03]{DK03} T. Dern, A. Krieg, \textit{Graded rings of Hermitian modular forms of degree 2.} Manuscr. Math. \textbf{110} (2003), 251–-272. 

\bibitem[DK06]{DK06} T. Dern, A. Krieg, \textit{The graded ring of Hermitian modular forms of degree 2 over $\QQ(\sqrt{-2} )$.} J. Number Theory \textbf{107} (2004), 241--265.

\bibitem[DKW19]{DKW19} C. Dieckmann, A. Krieg, M. Woitalla, \textit{The graded ring of modular forms on the Cayley half-space of degree two.} Ramanujan J. \textbf{48} (2019) 385--398.  


\bibitem[EZ85]{EZ85} M. Eichler, D. Zagier, \textit{The Theory of Jacobi
Forms.} Progress in Mathematics, vol. \textbf{55}. Birkh\"auser, Boston, Mass., 1985.

\bibitem[FH00]{FH00} E. Freitag, C. F. Hermann, \textit{Some modular varieties of low dimension.} Adv. Math. \textbf{152} (2000) 203--287.

\bibitem[FS07]{FS07} E. Freitag, R. Salvati Manni, \textit{Some modular varieties of low dimension II.} Adv. Math. \textbf{214} (2007) 132--145.

\bibitem[Gri18]{Gri18} V. Gritsenko, \textit{Reflective modular forms and their applications.} Russian Math. Surveys \textbf{73}:5 (2018), 797--864.

\bibitem[GHS09]{GHS09} V. Gritsenko, K. Hulek, G. K. Sankaran,  \textit{Abelianisation of orthogonal groups and the fundamental group of modular varieties.} J. Algebra \textbf{322} (2009), no.2, 463--478.

\bibitem[GHS13]{GHS13} V. Gritsenko, K. Hulek, G. K. Sankaran,  \textit{Moduli of $K3$ surfaces and irreducible symplectic manifolds,} from: "Handbook of moduli, I" (editors G Farkas, I Morrison), Adv. Lect. Math. 24, International Press (2013) 459--526.


\bibitem[GN18]{GN18} V. Gritsenko, V. V. Nikulin, \textit{Lorentzian Kac-Moody algebras with Weyl groups of $2$-reflections.} Proc. Lond. Math. Soc. (3) \textbf{116} (2018), no.3, 485--533.

\bibitem[Got69]{Got69} E. Gottschling, \textit{Invarianten endlichen Gruppen und biholomorphe Abbildungen.} Invent. Math. \textbf{6}:4 (1969), 315--326.

\bibitem[Gun63]{Gun63} K. B. Gundlach, \textit{Die Bestimmung der Funktionen zur Hilbertschen Modulgruppe des
Zahlkorpers $\QQ(\sqrt{5})$.} Math. Ann. \textbf{152} (1963), 226--256.

\bibitem[HU14]{HU14} H. Hashimoto,  K. Ueda, \textit{The ring of modular forms for the even unimodular lattice of signature $(2,10)$.} Preprint 2014, arXiv:1406.0332.

\bibitem[Igu62]{Igu62} J. Igusa, \textit{On Siegel modular forms of genus two.} Amer. J. Math. \textbf{84} (1962), 175--200.

\bibitem[Kri05]{Kri05} A. Krieg, \textit{The graded ring of quaternionic modular forms of degree 2.} Math. Z. \textbf{251} (2005) 929--944.

\bibitem[Ma17]{Ma17} S. Ma, \textit{Finiteness of $2$-reflective lattices of signature $(2,n)$.} Amer. J. Math. \textbf{139} (2017), 513--524.

\bibitem[Ma18]{Ma18} S. Ma, \textit{On the Kodaira dimension of orthogonal modular varieties.} Invent. math. \textbf{212} (2018), no. 3, 859--911. 

\bibitem[Nik80]{Nik80} V. V. Nikulin, \textit{Integral symmetric bilinear forms and some of their applications.} Math. USSR Izv. \textbf{14} (1980), no.1, 103--167.

\bibitem[PR94]{PR94} V. Platonov, A. Rapinchuk, \textit{Algebraic groups and number theory.} Pure and appl math. Vol 139, Academic Press 1994.

\bibitem[Run93]{Run93} B. Runge, \textit{On Siegel modular forms, I.}  J. Reine Angew. Math. (1993) \textbf{436}, 57--85.

\bibitem[Sch06]{Sch06} N. R. Scheithauer, \textit{On the classification of automorphic products and generalized Kac-Moody algebras.} Invent. Math. \textbf{164} (2006), 641--678.

\bibitem[Sch17]{Sch17} N. R. Scheithauer, \textit{Automorphic products of singular weight.}  Compos. Math. \textbf{153} (9), 1855--1892 (2017). 

\bibitem[Ser73]{Ser73} J. P. Serre, \textit{A Course in Arithmetic.} Graduate Texts in Mathematics, vol. 7, Springer-Verlag, New York, 1973.

\bibitem[ST54]{ST54} G. C. Shephard,  J. A. Todd, \textit{Finite unitary reflection groups.}  Canad. J. Math. \textbf{6} (1954), 274--304.

\bibitem[Stu19]{Stu19} E. S. Stuken, \textit{Free algebras of Hilbert automorphic forms.} Funkts. Anal. Prilohzen. \textbf{53}:1 (2019), 49--66; English transl.: Functional Anal. Appl., \textbf{53}:1 (2019), 37--50.

\bibitem[SV17]{SV17} O. V. Shvartsman, E. B. Vinberg, \textit{A criterion of smoothness at infinity for an arithmetic quotient of the future tube.} Funkts. Anal. Prilohzen. \textbf{51}:1 (2017), 40--59; English transl.: Functional Anal. Appl., \textbf{51}:1 (2017), 32--47.

\bibitem[Vin10]{Vin10}  E. B. Vinberg, \textit{Some free algebras of automorphic forms on symmetric domains of type IV.} Transform. Groups \textbf{15}:3 (2010), 701--741.

\bibitem[Vin13]{Vin13} E. B. Vinberg, \textit{On the algebra of Siegel modular forms of genus 2.} Trudy Moskov. Mat. Obshch. \textbf{74}:1, 1--16; English transl.: Trans. Moscow Math. Soc. \textbf{74} (2013), 1--13.

\bibitem[Vin18]{Vin18}  E. B. Vinberg, \textit{On some free algebras of automorphic forms.} Funkts. Anal. Prilohzen. \textbf{52}:4 (2018), 38--61; English transl.: Functional Anal. Appl., \textbf{52}:4 (2018), 270--289.

\bibitem[VP89]{VP89} E. B. Vinberg, V. L. Popov,  \textit{Invariant Theory,} in: Algebraic Theory-4, Itogi Nauki i Tekhniki. Sovrem. Probl. Mat. Fund. Napr., vol. 55, VINITI, Moscow, 1989, 137--309; English transl.: in: Encycl. Math. Sci., vol. 55, Algebrai Geometry-IV, Springer-Verlag, Berlin, 123--278.

\bibitem[Wan18]{Wan18} H. Wang,  \textit{Reflective Modular Forms: A Jacobi Forms Approach.}  Int.  Math.  Res.  Not.  IMRN.  \textbf{2021} no. 3 (2021),  2081--2107.

\bibitem[Wan18b]{Wan18b} H. Wang, \textit{Weyl invariant $E_8$ Jacobi forms.}  arXiv:1801.08462.

\bibitem[Wan19]{Wan19} H. Wang,  \textit{The classification of $2$-reflective modular forms.} arXiv:1906.10459.

\bibitem[Woi17]{Woi17} M. Woitalla, \textit{Theta type Jacobi forms.} Acta Arith. \textbf{181} (2017), 333--354.

\bibitem[WW20]{WW20} H. Wang, B. Williams, \textit{On some free algebras of orthogonal modular forms.} Adv. Math. \textbf{373} (2020), 107332, 22 pp.

\end{thebibliography}

\end{document}